\documentclass[11pt,twoside,dvipsnames]{amsart}
\usepackage[utf8]{inputenc}
\usepackage[english]{babel}
\usepackage[T1]{fontenc}
\usepackage{xcolor}
\usepackage{amssymb}
\usepackage{geometry}
\usepackage[bookmarks=true]{hyperref}
\usepackage{mathptmx}

\def\*#1{\relax\ifmmode\if\noexpand #1\relax \mathop{\kern
		0pt^*{#1}}\nolimits\else \kern 0pt^*\!#1 \fi\else$^*$#1\fi}

\newtheorem{Theorem}{Theorem}[section]
\newtheorem{Lemma}[Theorem]{Lemma}

\newtheorem{Proposition}[Theorem]{Proposition}

\theoremstyle{definition}

\newtheorem{Definition}[Theorem]{Definition}
\newtheorem{Example}[Theorem]{Example}
\def\NZQ{\mathbb}
\def\ZZ{{\NZQ Z}}
\def\FF{{\NZQ F}}
\def\RR{{\NZQ R}}
\def\mm{{\mathfrak m}}
\def\nn{{\mathfrak n}}
\def\pp{{\mathfrak p}}
\def\qq{{\mathfrak q}}
\def\Hom{\operatorname{Hom}}
\def\Ext{\operatorname{Ext}}
\def\Ann{\operatorname{Ann}}
\def\Ker{\operatorname{Ker}}
\def\Supp{\operatorname{Supp}}
\def\hht{\operatorname{height}}
\def\pd{\operatorname{projdim}}
\def\id{\operatorname{injdim}}
\def\rank{\operatorname{rank}}
\def\depth{\operatorname{depth}}
\let\iso=\cong
\let\Dirsum=\bigoplus
\let\phi=\varphi
\def\cM{{\mathcal M}}

\begin{document}
	
\title{$\ZZ^r$-graded rings and their canonical modules}
\author[M. Barile]{Margherita Barile}
\address{Margherita Barile\\ Universit\` a degli Studi di Bari Aldo Moro\\ Dipartimento di Matematica\\ 70125 Bari\\ Italy}
\email{margherita.barile@uniba.it}
\author[w. Bruns]{Winfried Bruns}
\address{Winfried Bruns\\ Universit\"at Osnabr\"uck\\ Institut für Mathematik\\ 49069 Os\-na\-br\"uck\\ Germany}
\email{wbruns@uos.de}

\begin{abstract}
In ``Cohen--Macaulay rings'' Bruns and Herzog define the graded canonical module for $\ZZ$-graded rings. We generalize the definition to multigradings and prove that the canonical module ``localizes''. As an application, we give a divisorial proof of the theorem of Danilov and Stanley on the canonical module of affine normal monoid rings. Along the way, we develop the basic theory of multigraded rings and modules.
\end{abstract}

\keywords{multigrading, canonical module, normal affine monoid ring}

\subjclass[2020]{13C14, 13C15, 13C70, 13F05}
\maketitle

The original motivation of this article was to find a divisorial proof of the theorem of Danilov and Stanley that describes the canonical module of normal affine monoid rings (for example, see Bruns and Herzog \cite[Chapter 6]{BH}). It becomes readily clear that such a proof requires that the multigraded canonical module ``localizes''.  This will be proved in Section \ref{canonical}, after  the introduction of the $\ZZ^r$-graded canonical module. The application to graded canonical modules of normal affine monoid rings is given in Section \ref{DanStan}.

As in \cite{BH}, we use the \*notation of Fossum and Foxby \cite{FF} that indicates properties and invariants defined in the category of graded rings and modules. A \*local graded ring is a graded ring whose homogeneous nonunits generate a proper ideal. For $\ZZ$-graded \*local rings the \*canonical module was introduced by Jürgen Herzog in \cite[Section 3.6]{BH}. Jürgen  died in April 2024. The canonical module was the theme of his life, from the 1971 lecture notes \cite{HK} with Kunz to 2023 preprints in \url{arXiv.org}.  This article is a tribute to J\"urgen, a master of commutative algebra, a teacher for the first author, for the second a coauthor and friend.

We use the opportunity to give an expository style introduction to $\ZZ^r$-graded rings in Sections \ref{basic}--\ref{HomProps}  , following the development in \cite{BH}. This includes Krull dimension and chains of prime ideals, results of Matijevic--Roberts type \cite{MR}, as well as homological properties and invariants. Several of these results can also be found in Goto and Watanabe \cite{GW}. We do not claim originality for them. We have not included local duality. It is left to the reader. Generalizing it from $\ZZ$-gradings in \cite[Section 3.6]{BH} to $\ZZ^r$-gradings is an excellent training.

For unexplained terminology and symbols, we refer the reader to \cite{BH}.

\section{Basic notions and results}\label{basic}

For readers that are not familiar with the category of $\ZZ^r$-graded modules we briefly recapitulate some definitions. Let $R$ be a $\ZZ^r$-graded ring: $R =\bigoplus_{g\in\ZZ^r} R_g$, where the homogeneous components $R_g$ are $R_0=R_{(0, \dots, 0)}$-modules and $R_gR_h\subset R_{g+h}$ for all $g, h\in\ZZ^r$. A $\ZZ^r$-graded $R$-module $M$ has a decomposition $M=\bigoplus_{h\in \ZZ^r} M_h$ such that $R_gM_h\subset M_{g+h}$ for all $g,h\in \ZZ^r$. The elements of $R_g$ and $M_g$, respectively, have \emph{degree $g$}. Homomorphisms of $\ZZ^r$-graded rings and modules map homogeneous elements to homogeneous elements and preserve degrees. A submodule is graded if it is generated by homogeneous elements. It is easy to see that images and preimages of graded submodules are graded as well as cokernels. In other words: the graded modules with the homogeneous homomorphisms form an abelian category $\cM_0(R)$ (see \cite{Ro} for the homological algebra we will use).

In the following we will often say ``graded'' instead of ``$\ZZ^r$-graded''. If the restriction to $r=1$ is necessary, we will use ``$\ZZ$-graded''. Often ``homogeneous'' is a synonym for ``graded''.

Let $M$ be a graded $R$-module, generated by a family $(x_g)_{g\in G}$ of homogeneous elements. Set $d_g =\deg g$ for $g\in G$. Consider the free $R$-module
$$
F = \Dirsum_{g\in G}R(-d_g)
$$
and define the $R$-homomorphism $\phi: F \to M$ as the sum of the linear maps $R(-d_g)\to M$, $1\mapsto x_g$. Then $\phi$ is homogeneous, and so the kernel $\Ker \phi$ is again graded. Applying the same construction to $\Ker´\phi$ and continuing, one constructs a graded free resolution of $M$. It follows that a  projective object in the category $\cM_0(R)$ is a projective $R$-module. Moreover $\cM_0(R)$ has enough projectives. It has enough injectives as well, but their construction is more involved since an injective object in $\cM_0(R)$ need not be an injective module. See \cite[Section 3.6]{BH} for the construction of injective resolutions in $\cM_0(R)$.

Let $R$ be a $\ZZ^r$-graded ring and $M$ a $\ZZ^r$-graded  $R$-module.
For all $g=(g_1,\dots, g_r)\in\ZZ^r$, we set $\delta_i(g)=g_i$. Let $x\in M$. For all $i\in\lbrace{1,\dots, r\rbrace}$ and all $u\in\ZZ$ we  set 
$$
x^{(i)}_u=\sum_{\delta_i(g)=u}x_g.
$$
Here $x_g$ denotes the homogeneous component of $x$ of degree $g$. The element $x$ is \emph{$\delta_i$-homogeneous} if $x = x^{(i)}_u$ for some $u$. A submodule  $N$ of $M$ is \emph{$\delta_i$-homogeneous} if, for all $x\in N$ and all $u\in\ZZ$, $x^{(i)}_u\in N$.

\begin{Lemma}\label{delta-hom}
The $R$-submodule $N$ of $M$  is $\ZZ^r$-graded if and only if $N$ is $\delta_i$-homogeneous for all $i=1,\dots,r$.
\end{Lemma}

\begin{proof}
The ``only if'' part is trivial. For the ``if'' part, we introduce the following notation. Given an integer $s$ with $1\leq s\leq r$, and an $s$-tuple $(g_1,\dots, g_s)\in \ZZ^s$, for any $x\in M$ we set
$$x_{(g_1,\dots, g_s)}=\sum_{{\delta_i(g)=g_i}\atop{1\leq i\leq s}}x_g.$$
Suppose that $N$ is $\delta_i$-homogeneous for all $i=1,\dots, r$. Let $x\in N$. We prove that, for all integers $s$ with $1\leq s\leq r$, and any $s$-tuple $(g_1,\dots, g_s)\in \ZZ^s$, $x_{(g_1,\dots, g_s)}\in N$. We proceed by finite induction on $s$. The case $s=1$ is provided by the assumption of $\delta_1$-homogeneity. So assume that $1<s\leq r$ and that the claim is true for $s-1$. 
Let $(g_1,\dots, g_{s-1}, g_s)\in\ZZ^s$. By the induction hypothesis,  $x_{(g_1,\dots, g_{s-1})}\in N$. Hence 
$$x_{(g_1,\dots, g_{s-1}, g_s)}={x_{(g_1,\dots, g_{s-1})}^{(s)}}_{g_s}\in N$$
follows by $\delta_s$-homogeneity. 
\end{proof}

Given an ideal $I$ of $R$, we denote by $I^*$ the ideal of $R$ generated by the $\ZZ^r$-homogeneous components of the elements of $I$. Clearly, $I^*$ is the greatest $\ZZ^r$-graded ideal contained in $I$. It will be used in the following that $I^*$ can be constructed step by step where each step requires taking homogeneous elements with respect to a $\ZZ$-grading.

\begin{Lemma}\label{step_by_step}
Let $J$ be the ideal generated by the elements of $I$ that are $\delta_i$-homogeneous for all $i=1,\dots,r-1$. Then $I^*$ is the ideal generated by the $\delta_r$-homogeneous elements in $J$.
\end{Lemma}

The lemma is proved by similar arguments as Lemma \ref{delta-hom}.

\begin{Lemma}\label{Laurent} Let $R$ be a $\ZZ^r$-graded ring. Then the following conditions are equivalent:
\begin{enumerate}
\item[(a)] every nonzero homogeneous element of $R$ is invertible;
\item[(b)] $R_0=k$ is a field and either $R=k$ or $ R=k[t_1,\dots, t_s, t_1^{-1}, \dots, t_s^{-1}]$ for some integer $s$ such that $1\leq s\leq r$, where $t_1,\dots, t_s$ are algebraically independent elements over  $k$ of nonzero degree.
\end{enumerate}
\begin{proof}
The implication (b) $\implies$ (a) is trivial.  We prove the implication (a) $\implies$ (b).  By assumption, $R_0=k$ is a field. If $R=R_0$, then $R=k$. Otherwise $R$ contains at least one homogeneous element of nonzero degree. Since all nonzero homogeneous elements of $R$ are invertible, their degrees form a subgroup of $\ZZ^r$, of rank $s\leq r$. Let $t_1,\dots, t_s$ be homogeneous elements of $R$ such that their degrees $d_i=\deg t_i$ form a basis of this subgroup. We now consider the ring $S=k[T_1,\dots, T_s, T_1^{-1},\dots, T_s^{-1}]$  of Laurent polynomials in $s$ variables, where $\deg(T_i)=d_i$ for all $i=1,\dots, s$.  We then define the $\ZZ^r$-graded ring homomorphism $\phi: S \to R$ that maps $k$ identically to $R_0$ and $T_i$ to $t_i$ for all $i=1,\dots, s$. We claim that $\phi$ is an isomorphism. For injectivity, consider an element $f$ of $\mathrm{Ker}\,\phi$, say
$$
f=\sum_{g\in\ZZ^r}a_gT_1^{\delta_1(g)}\cdots T_s^{\delta_s(g)},
$$
where, for all $g$, $a_g$ is an element of $k$. Then
$$
0=\phi(f)=\sum_{g\in\ZZ^r}a_gt_1^{\delta_1(g)}\cdots t_s^{\delta_s(g)},
$$
which implies that $a_gt_1^{\delta_1(g)}\cdots t_s^{\delta_s(g)}=0$ for all $g\in\ZZ^r$. Since all elements $t_i$ are invertible, it immediately follows that $a_g=0$, for all $g\in\ZZ^r$. Thus $f=0$, which proves that $\phi$ is injective. 

For surjectivity, consider a nonzero homogeneous element $a\in R$ of degree $d$. Let $c_1, \dots, c_s$ be integers such that $d=\sum_{i=1}^sc_id_i$. Then $\deg(a\prod_{i=1}^st_i^{-c_i})=0$, which means that $a=b\prod_{i=1}^st_i^{c_i}$ for some $b\in k$. Thus $a=\phi(b\prod_{i=1}^sT_i^{c_i})\in\,\mathrm{Im}\,\phi$. 
\end{proof}
\end{Lemma}

In the sequel, the number $s$ occurring in the previous proof, i.e., the rank of the subgroup of $\ZZ^r$ generated by the degrees of the nonzero elements of $R$, will be denoted by $\sigma(R)$.

Since it is important in the following, we state a trivial fact as a lemma.
The proof is left to the reader.

\begin{Lemma}\label{iso}
Let $M$ be a graded module over the graded ring $R$. If R contains a unit of degree $i$, then $M\iso M(i)$ as graded modules.
\end{Lemma}

Almost never a localization $R_\pp$ of a graded ring $R$ with respect to a prime ideal $\pp$ is graded. We stay in the category of graded rings and modules if we take the \emph{homogeneous localization} $R_{(\pp)} = T^{-1}R$ where $T$ is the set of homogeneous elements of $R$ outside $\pp$. For a graded module we set $M_{(\pp)} = T^{-1}M$.

\section{Dimension theory} 

In this section we discuss the relationship between a prime ideal $\pp$ in a graded ring and the ideal $\pp^*$ generated by the homogeneous elements in $\pp$, which turns out to be a prime ideal itself.

\begin{Lemma}\label{star}\leavevmode
\begin{enumerate}
\item[(a)] If $\pp$ is a prime ideal of $R$, then $\pp^*$ is a prime ideal. 
\item[(b)]
\begin{enumerate}
\item[(i)] If $\pp\in\Supp\, M$, then $\pp^*\in\Supp\, M$.
\item[(ii)]  Let $M$ be a $\ZZ^r$-graded $R$-module. Suppose that $\pp\in\mathrm{Ass}(M)$. Then $\pp$ is  $\ZZ^r$-graded. Furthermore, $\pp$ is the annihilator of some homogeneous element. 
\end{enumerate}
\end{enumerate}
\end{Lemma}

\begin{proof}
The crucial point is that $\ZZ^r$ is an ordered group with respect to a suitable order, for example the lexicographic order. The relation $<$ for elements of $\ZZ^r$ refers to it, as well as the attributes ``minimal'' and ``maximal''.

(a) Let $a,b\in R$ be such that $ab\in \pp^*$.  Suppose for a contradiction that $a\notin \pp^*$ and $b\notin \pp^*$. Then, for some  $u, v\in\ZZ^r$, $a_u\notin \pp^{\ast}$ and $b_v\notin \pp^*$. Let $u, v$ be minimal. The homogeneity of $\pp^*$ implies that $(ab)_{u+v}\in \pp^*$. But $(ab)_{u+v}=\sum_{i+j=u+v}a_ib_j$ and all summands, possibly up to $a_ub_v$, belong to $\pp^*$. This implies that $a_ub_v\in \pp^*\subset \pp$. Since $\pp$ is prime, it follows that $a_u\in \pp^{\ast}$ or $b_v\in \pp^*$, a contradiction.

(b) Suppose for a contradiction that $\pp^*\notin\Supp\, M$. Then $M_{\pp^*}=0$. Let $x\in M$ be a homogeneous element. Then there is an element $a\in R\setminus\pp^*$ such that $ax=0$. Since $x$ is homogeneous, it follows that $a_gx=0$ for all $g\in\ZZ^r$. On the other hand,  $a_g\notin \pp^*$ for some $g\in\ZZ^r$. Since $a_g$ is homogeneous, we even have that $a_g\notin\pp$. Hence $x/1=0$ in $M_{\pp}$. But this is true for any homogeneous $x\in M$. Thus $M_{\pp}=0$, against our assumption. 

We now prove (ii). Let $x$ be an element of $M$ such that $\pp=\Ann x$ and pick $a\in \pp$. We expand $a$ and $x$ into homogeneous components:
$$
a = a_{g_0}+\dots+a_{g_s} , \qquad x = x_{h_0} +\dots + x_{h_t},
$$
with $g_0 < \dots < g_s$ and $h_0 < \dots < h_t$. We claim
$$
a_{g_0}^{k+1} x_{h_k} = 0,\qquad k=0,\dots,t.
$$
Then $a_{g_0}^{t+1}$ annihilates $x$, and so $a_{g_0}^{t+1}\in \pp$. Since $\pp$ is a prime ideal, $a_{g_0}\in\pp$. It follows that $a_{g_0}$ annihilates all homogeneous components of $x$. Moreover, 
$(a_{g_1}+\dots+a_{g_s})x = 0$. By induction on $s$ it follows that all homogeneous components of $a$ lie in $\pp$, and, moreover, $a$ annihilates all homogeneous components of $x$. Hence, $\pp\subset\bigcap_{k=0}^t\Ann x_{h_k}$. On the other hand, the opposite inclusion is obvious. The primality of $\pp$ thus implies that $\pp=\Ann x_{h_k}$ for some $k$. 

In order to prove the claim above, we split the set of pairs $(i,j)$, $i=0,\dots,s$, $j=0,\dots,t$, into subsets $G_w =\{(i,j): g_i + h_j = w\}$, $w\in \ZZ^r$. The smallest $w$ for which $G_w\neq \emptyset$ is $w = g_0+ h_0$, and $G_w = \{(g_0, h_0) \}$. So $ax = 0$ implies $a_{g_0} x_{h_0} = 0$. This starts an induction on $k$. Now take $w = g_0 + h_k$. Then all pairs  $(i,j) \in G_w$, $j\neq k$ satisfy $j < k$. We have
$$
\sum_{(i,j)\in G_w} a_{g_i}x_{h_j} = 0.
$$
Multiplying by $a_{g_0}^k$ and using the induction hypothesis, we see that $a_{g_0}^{k+1} x_{h_k} = 0$.
\end{proof} 

Lemma \ref{star} is the starting point for the construction of chains of graded prime ideals in a graded ring and the comparison of $\pp$ and $\pp^*$.

\begin{Theorem}\label{chain} 
Let $R$ be a Noetherian $\ZZ^r$-graded ring, $M$ a finite $\ZZ^r$-graded $R$-module and $\pp\in\,\Supp\, M$. 
\begin{enumerate}
\item[(a)] If $\pp$ is $\ZZ^r$-graded, then there exists a chain  $\pp_0\subset\cdots\subset\pp_d=\pp$, $d=\dim M_{\pp}$,  of $\ZZ^r$-graded prime ideals $\pp_i\in\,\Supp\,M$. 
\item[(b)] If $\pp$ is not $\ZZ^r$-graded, then $\dim M_{\pp^*}<\dim M_{\pp}\leq \dim M_{\pp^*}+\sigma(R).$
\end{enumerate}
\end{Theorem}

\begin{proof} We first assume that $\pp$ is not $\ZZ^r$-graded.  Then we immediately have $\dim M_{\pp^*}<\dim M_{\pp}$.  Next we show that $\hht \pp/\pp^*\leq \sigma(R)$. We first move to the quotient ring $S = R/\pp^*$, and consider the ideal $\pp/\pp^*.$ It does not contain any nonzero homogeneous element. If we invert  these elements, we obtain the homogeneous localization $S_{(0)}$. Now, $\pp S_{(0)}$ is a nonzero prime ideal, and, by Lemma \ref{Laurent}, the ring $S_{(0)}$ is of the form $k[T_1,\dots, T_s, T_1^{-1}, \dots, T_s^{-1}]$ for some positive integer $s \le \sigma(R)$. Since the height of ideals in the Laurent polynomial ring is bounded by $s$, the required inequality follows. If $R$ is catenary, the inequality in (b) follows. But we will prove it independently of catenarity.
	
The proof of (a) is by induction on  $r$, starting  with the trivial case $r=0$. Set $s=r-1$. Since $\pp$ is $\ZZ^s$-graded, there exists a chain  $\pp_0\subset\cdots\subset\pp_d=\pp$, $d=\dim M_{\pp}$ of $\ZZ^s$-graded prime ideals. By Lemma \ref{star}, $\pp_0$ is $\ZZ^r$-graded, and this covers the case $d=1$. Assume that $d>1$.  if $\pp_1$ is $\ZZ^r$-graded or we can replace it by a $\ZZ^r$-graded prime ideal, we are done by induction after passing to $M/\pp_1M$. Assume that $\pp_1$ is not $\ZZ^r$-graded, and let $\qq$ be the ideal generated by the  $\delta_r$-homogeneous elements in $\pp_1$. Then, by what has been shown above and by Lemma \ref{step_by_step}, $\qq$ is $\ZZ^r$-graded and lies properly between $\pp_0$ and $\pp_2$. 

It remains to show that $\dim M_{\pp}\leq \dim M_{\pp^*}+\sigma(R)$. Only if $\pp$ is not graded there is something to show.  Let $r=1$ first. There exists a chain  $\pp_0\subset\cdots\subset\pp_d=\pp$, $d=\dim M_{\pp}$,  of  prime ideals $\pp_i\in\Supp M$. Since $\pp_0$ is graded, we can apply induction on $d$ and assume that $\pp_0,\dots,\pp_{d-2}$ are $\ZZ^r$-graded. Evidently $\pp_{d-2} \subset\pp^*\subset \pp_d$, and both inclusions are strict  since $\hht \pp/\pp^* =1$.
	
For general $r$, we use Lemma \ref{step_by_step}. Set $\pp^{(0)} = \pp$ and let $\pp^{(i)}$  be the ideal generated by the elements in  $\pp^{(i-1)}$ that are $\delta_i$-homogeneous, $i = 1,\dots,r$. Then $\dim M_{\pp^{(i-1)}}\le \dim M_{\pp^{(i)}} +1$, and we are done by accumulating the inequalities.
\end{proof}
 
In view of Theorem \ref{chain}(b) we introduce the invariant 
$$
\tau(\pp) = \dim R_\pp - \dim R_{\pp^*}.
$$
Clearly, with the notation of the proof of Theorem \ref{chain}, $\tau(\pp) = \vert\{i: \pp^{(i)} \neq \pp^{(i-1)}\}\vert$. One should notice that the number is independent of base changes in $\ZZ^r$.

\section{Homological invariants}\label{HomProps}

We need an operation that keeps the structure of an $R$-module, but changes the degrees: the degree $g$ homogeneous component of $M(s)$ is $M_{g+s}$. In this context the integer $r$-tuple $s$ is called a \emph{shift}. The $R$-homomorphisms $f:M\to M'$ where $M$ and $M'$ are $\ZZ^r$-graded $R$-modules, satisfying $f(M_g)\subset (M')_{g+s}$ for all $g$ form an $R_0$-module $\Hom_s(M,M')\subset \Hom_R(M,M')$. One sets
$$
\*\Hom_R(M,M') = \Dirsum_s \Hom_s(M,M') \subset \Hom_R(M,M').
$$

Note that $\*\Hom_R(M,M')\neq \Hom_R(M,M')$ in general, but equality holds if $M$ is finitely generated and $R$ is Noetherian since $\*\Hom$ is left exact and equality certainly holds for free modules of finite rank. The derived functors  of $\*\Hom$ are called $\*\Ext$. It is easy to see that $\*\Hom(M(u), N(v)) = \*\Hom(M,N)(-u+v)$ for all $u,v\in \ZZ^r$, and the same rule applies to $\*\Ext$.

For a Noetherian ring $R$ and a finitely generated $\ZZ^r$-graded $R$-module $M$, one has $\*\Ext_R^i(M,N) = \Ext_R^i(M,N)$ for all graded $R$-modules $N$ and all $i$ so that $\*$ only emphasizes the graded structure. 

\emph{Rees' lemma} is an important tool by which one can control the effect of going modulo a regular sequence:

\begin{Lemma} \label{Rees}
\index{Rees' lemma} Let $R$ be a $\ZZ^r$-graded ring, and $Q$ and
$N$  $\ZZ^r$-graded $R$-modules. If $x\in R$ is a homogeneous element of
degree $g$ which is a non-zerodivisor of $R$ and $Q$ and annihilates
$N$, then $\*\Ext_R^{i+1}(N,Q)\iso\*\Ext_{R/(x)}^i(N,Q/xQ)(g)$ for all $i$.
\end{Lemma}

The nongraded version is proved in \cite[3.1.16]{BH}, and the
graded version is proved in the same way, using the axiomatic
characterization of derived functors.

Let $R$ be a Noetherian ring, and $\pp$ a prime ideal. By $k(\pp)$ we denote the residue class field of the localization $R_\pp$, $k(\pp) = R_\pp/\pp R_\pp$.  One calls  $\mu_i(\pp,M) = \dim_k \*\Ext_R^i(k(\pp), M_\pp)$ the $i$-th \emph{Bass number} of $M$ at $\pp$. 
By definition  the depth of a finite nonzero module $M$ over a local Noetherian ring $R$ with maximal ideal $\mm$ is the maximal length of an $M$-regular sequence in $\mm$. Homologically, it is characterized by
$$
\depth M = \min\bigl\{i: \*\Ext_R^i(k, M)\neq 0   \bigr\}, \qquad k= R/\mm.
$$
The  \emph{type} of $M$ is $r(M) = \dim_k \*\Ext_R^t(k, M)$, $t = \depth M$.  The Bass numbers at $\mm$ determine the \emph{injective dimension $\id $}, the minimal length of a minimal injective resolution of $M$,
$$
\id M = \sup \bigl\{i: \*\Ext_R^i(k, M)\neq 0   \bigr\}.
$$
In the following we want to relate these invariants for the localizations with respect to $\pp$ and $\pp^*$. The crucial case is that of $\ZZ$-gradings, for which we copy the proof \cite[1.5.9]{BH}.

\begin{Theorem}\label{invariants} 
Let $R$ be a Noetherian $\ZZ$-graded ring, $M$ a finite $\ZZ$-graded $R$-module and $\pp\in\,\Supp\,M$ a nongraded prime ideal. Then
\begin{enumerate}
\item[(a)]
$\depth M_{\pp}= \depth M_{\pp^*}+1$,
\item[(b)] $ r(M_{\pp})= r(M_{\pp^*})$,
\item[(c)]  $\mu_o(\pp, M) = 0$ and $\mu_{i+1}(\pp, M)=\mu_i(\pp^*, M)$ for all $i \ge 0$,
\item[(d)] $\id_{R_\pp} M_\pp = \id_{R_{\pp^*}} M_{\pp^*} + 1$,
\item[(e)] $\pd_{R_\pp} M_\pp = \pd_{R_{\pp^*}}M_{\pp^*}$.
\end{enumerate}
\end{Theorem}

\begin{proof}
For (a), we may consider $M_{\pp}$ and $M_{\pp^*}$ as modules over the homogeneous localization $R_{(\pp)}$ and replace $R$ by it. Thus we may assume that $R/\pp^*\simeq k[t, t^{-1}]$, where $k$ is a field and $t$ is an element of positive degree that is transcendental over $k$. Since this ring is a PID, there is some $a\in R\setminus\pp$ such that $\pp=\pp^*+(a)$. Hence we have a short exact sequence: 
$$
0\to R/\pp^*\stackrel{a}{\to} R/\pp^*\to R/\pp\to 0,
$$
which gives rise to the following long exact sequence:
$$
\cdots\to \*\Ext^i_R(R/\pp^*,M)\stackrel{a}{\to}\*\Ext^i_R(R/\pp^*,M)\to\\*\Ext^{i+1}_R(R/\pp,M)\to\cdots
$$

Now, for all $i$, $\*\Ext^i_R(R/\pp^*,M)$ is a $\ZZ$-graded $R/\pp^*=k[t, t^{-1}]$-module, and, therefore, is free. Consequently, since $a\notin\pp^*$, the map
$$
\*\Ext^i_R(R/\pp^*,M)\stackrel{a}{\to}\*\Ext^i_R(R/\pp^*,M)
$$
is injective, whence
$$
\Ext^{i+1}_R(R/\pp,M)\simeq\*\Ext^i_R(R/\pp^*,M)/a\cdot \*\Ext^i_R(R/\pp^*,M).
$$

On the other hand, $\pp=\pp^*+(a)$ implies that $\Ext^{i+1}_R(R/\pp,M)$ is a free $(R/\pp)$-module of the same rank as the free ($R/\pp^*)$-module $\*\Ext^i_R(R/\pp^*,M)$. We deduce that
\begin{multline*}
\dim_{k(\pp)}\Ext^{i+1}_{R_{\pp}}(k(\pp),M_{\pp})=
\rank_{R/\pp}\Ext^{i+1}_R(R/\pp, M)\\
 = \rank_{R/\pp^*}\*\Ext^i_R(R/\pp^*, M)=\dim_{k(\pp^*)}\*\Ext^i_{R_{\pp^*}}(k(\pp^*),M_{\pp^*}),
\end{multline*}
where, for better readability, we have set $k(\pp) = R_\pp/\pp R_\pp$ and $k(\pp^*) = R_{\pp^*}/{\pp^*} R_{\pp^*}$. This equation proves (a)--(d).

For (e), one considers a graded free resolution 
$$
\FF:\cdots \to F_{i+1}\stackrel{\phi_{i+1}}{\to} F_i\to \dots\to F_1\stackrel{\phi_{1}}{\to} F_0
$$
 of $M$. In order to compute $\pd_{R_\pp}$, we localize the resolution and check at which homological degree it splits. This homological degree is determined by the ideals $I_{r_j}(\phi_j)$ of minors of size $r_j= \rank \phi_j$: $\pd_{R_\pp} M_\pp = n$ if and only $I_{r_{n+1}}(\phi_{n+1})\not\subset \pp$. (see \cite[Section 1.4]{BH} for the linear algebra involved).  Since the free resolution is graded, the minors generating $J = I_{r_{n+1}}(\phi_{n+1})$ are homogeneous. Thus $J\not\subset \pp$ if and only if $J\not\subset \pp^*$.
\end{proof}

Applying the inductive argument presented in the proof of Theorem \ref{chain}, we generalize Theorem \ref{invariants} to $\ZZ^r$-gradings. 

\begin{Theorem}\label{depth-type} 
Let $R$ be a Noetherian $\ZZ^r$-graded ring, $M$ a finite $\ZZ^r$-graded $R$-module and $\pp\in\,\Supp\,M$ a prime ideal. Then
\begin{enumerate}
\item[(a)] $\depth M_{\pp}- \depth M_{\pp^*}=\dim\,M_{\pp}- \dim\,M_{\pp^*} =\tau(\pp)$. 
\item[(b)] $r(M_{\pp})= r(M_{\pp^*}).$
\item[(c)] If $\pp$ is nongraded, then $\mu_i(\pp, M) = 0$ for $i < \tau(\pp)$ and $\mu_{i+\tau(\pp)}(\pp, M)=\mu_i(\pp^*, M)$ for every nonnegative integer $i$.
\item[(d)] $\id_{R_\pp} M_\pp = \id_{R_{\pp^*}} M_{\pp^*} + \tau(\pp)$.
\item[(e)] $\pd_{R_\pp} M_\pp = \pd_{R_{\pp^*}} M_{\pp^*}$.
\end{enumerate}
\end{Theorem}

From Theorem \ref{depth-type}, we deduce that, given a finite $\ZZ^r$-graded $R$-module $M$ and a prime ideal $\pp$ of $R$, the localized module $M_{\pp}$ is Cohen--Macaulay if and only if this is true for $M_{\pp^*}$. Since a local ring is Gorenstein if and only if it is Cohen--Macaulay of type $1$, it follows that the local ring $R_{\pp}$ is Gorenstein if and only if this is true for $R_{\pp^*}$.  Now, since $M$ is Cohen--Macaulay and $R$ is Gorenstein if and only if so are all their localizations at prime ideals, we obtain the following characterizations.

\begin{Theorem}\label{CM} Let $R$ be a Noetherian $\ZZ^r$-graded ring, and $M$ a finite $\ZZ^r$-graded $R$-module. Then the following conditions are equivalent:
\begin{enumerate}
\item[(a)] $M$ is Cohen--Macaulay.
\item[(b)] For all prime ideals $\pp$, the  $R_{\pp^*}$-module $M_{\pp^*}$ is Cohen-Macaulay.
\end{enumerate}
\end{Theorem}

Theorem \ref{CM} was proved for $r=1$ by Matijevic--Roberts \cite{MR} using different arguments, answering a question of Nagata.

\begin{Theorem}\label{Gorenstein} Let $R$ be a Noetherian $\ZZ^r$-graded ring. Then the following conditions are equivalent:
\begin{enumerate}
\item[(a)] $R$ is Gorenstein.
\item[(b)] For all prime ideals $\pp$, the ring $R_{\pp^*}$ is Gorenstein.
\end{enumerate}
\end{Theorem}

Below we will also prove a variant of Theorem \ref{Gorenstein} for regularity.

Let $M$ be a graded module. In the category $\cM_0(R)$ we can define its projective dimension $\*\pd M$ and the injective dimension $\*\id M$ by the infimum of the lengths of projective and injective resolutions in $\cM_0(R)$. As pointed out in Section \ref{basic}, a graded module $M$ is a projective object of $\cM_0(R)$ if and only if it is a projective module in the plain sense. Thus $\*\pd M = \pd M$ for all modules $M$. We have also mentioned that the analogous statement for injective modules is not true. However, $\*\id M$ can be measured by the Bass numbers with respect to graded prime ideals; see \cite[3,6.4]{BH}. Using Theorem \ref{depth-type}, we thus get:

\begin{Theorem}\label{pd_id_r}
Let $R$ be a $\ZZ^r$-graded ring and $M$ a  $\ZZ^r$-graded $R$-module. Then
\begin{enumerate}
\item[(a)] $\pd _R M= \*\pd _R M$.
\item[(b)] $\id _R M \leq \*\id _R M+\sigma(R)$.
\end{enumerate} 
\end{Theorem} 

 A different proof (b) is provided in \cite[1.3.6]{GW}. \bigskip

\noindent\emph{\*Local rings}\enspace The graded counterparts of local rings are called \emph{$\*local$}: a graded ring is \*local if its homogeneous nonunits generate a proper ideal. Equivalently, the set of graded ideals has a unique maximal element with respect to inclusion. Usually we denote this \emph{\*maximal} ideal by $\mm$. The \emph{\*dimension} of a graded $R$-module $M$ is given by $\dim M_\mm$.

Let $R$ be a $\ZZ$-graded ring for which $R_0$ is a field, and whose semigroup of degrees, $\bigl\{\deg x: x\in R \text{ homogeneous} \bigr\}$ has no invertible element $\neq 0$. Such a \emph{positively graded} ring is the prototype of a \*local ring for which the \*maximal ideal is an ordinary maximal ideal. On the opposite side, we have a Laurent polynomial ring over a field whose \*maximal ideal is zero. For a prime ideal $\pp$ of a $\ZZ^r$-graded ring, the homogeneous localization $R_{(\*\pp)}$ is \*local with \*maximal ideal $\pp^*R_{(\*\pp)}$.

It is an important fact that the main invariants of finite graded $R$-modules $M$ over \*local Noetherian rings can be computed by localization. In particular this is true for \emph{minimal systems of generators}, i.e., homogeneous systems of generators of $M$ that are minimal with respect to inclusion. Let $x_1,\dots,x_n$ be such a system, and assume that
$$
a_1x_1 + \dots + a_nx_n = 0, \qquad a_1,\dots,a_n \in R.
$$
Splitting the $a_i$ into their homogeneous components, we can decompose the equation degree by degree. The minimality of $x_1,\dots,x_n$ implies that the homogeneous components of the $a_i$, and therefore the $a_i$ themselves, belong to the \*maximal ideal $\mm$. In more abstract terms: let $F_0 =\Dirsum_{i=1}^n R(-\deg x_i)$ and define the linear map $\phi_0: F_0\to M$ by sending the basis element $1$ of the $i$-th summand to $x_i$. Then $U_1 = \Ker \phi_0$ is again graded and contained in $\mm F_0$. Applying the same construction to $U_1$ and continuing, one obtains a \emph{minimal graded free resolution} $\FF$ of $M$. It follows that $\FF_\mm$ is a minimal free resolution of $M_\mm$ over the local ring $R_\mm$.

\begin{Theorem}\label{*local}
Let $M$ be a finite graded module over the \*local Noetherian ring $R$ with \*maximal ideal $\mm$. Then the following hold:
\begin{enumerate}
\item[(a)] A homogeneous minimal system of generators of $M$ is also a minimal system of generators of the localization $M_\mm$ over $R_\mm$.
\item[(b)] If $\FF$ is a minimal graded free resolution of $M$ over $R$, then $\FF_\mm$ is a minimal free resolution of $M_\mm$ over $R_\mm$.
\item[(c) ] $\pd _R M = \pd_{R_\mm} M_\mm$.
\item[(d)] If $M$ is projective, then it is free.
\end{enumerate}
\end{Theorem} 

The crucial statement is (a), which we have verified above. Then (b) follows immediately, as we have also seen. Since the projective dimension is the length of a minimal free resolution, (c) is an obvious consequence. For (d), one uses that a minimal homogeneous system of generators becomes linearly independent upon localization since $M_\mm$ is free. But then it was linearly independent from the beginning.

Theorem \ref{*local} can be extended: all statements in \cite[1.5.15]{BH}  hold for $\ZZ^r$-graded \*local rings as well. Moreover, one can show that minimal graded free resolutions are unique up to isomorphism and define graded Betti numbers. They belong to the most important numerical invariants of combinatorial commutative algebra.

We now turn to the study of the regularity of rings. Recall that a Noetherian local ring $R$ with maximal ideal $\mm$ is called \textit{regular} if $\mm$ is generated by $\dim R$ elements. By the theorem of Auslander--Buchsbaum--Serre, $R$ is regular if and only if all finite $§R$-modules have finite projective dimension. It is sufficient that $\pd_R R/\mm<\infty$. Let us first relate the regularity of $R_\pp$ and $R_{\pp^*}$.

\begin{Proposition}
Let $R$ be a Noetherian $\ZZ^r$-graded ring and $\pp$ a prime ideal. Then $R_\pp$ is regular if and only if $R_{\pp^*}$ is regular.
\end{Proposition}

\begin{proof}
Our standard induction argument allows us to assume that $r=1$.
The regularity of $R_{\pp^*}$ follows from that of $R_\pp$ immediately from the Auslander--Buchsbaum--Serre theorem. For the converse implication, we first replace $R$ by the homogeneous localization $R_{(\pp)}$. So we may assume that $R$ is \*local. Again we use the short exact sequence 
$$
0\to R/\pp^*\stackrel{a}{\to} R/\pp^*\to R/\pp\to 0
$$
that has appeared in the proof of Theorem \ref{invariants}. The assumption is that $R_{\pp^*}/ \pp^* R_{\pp^*}$ has finite projective dimension over $R_{\pp^*}$. But this implies $\pd_R  R/\pp^* < \infty$ by Theorem \ref{*local}(c). The exact sequence above then yields $\pd_R R/\pp \le \pd_R  R/\pp^* + 1 <\infty$. Finally, localizing at $\pp$ shows the regularity of $R_\pp$.   
\end{proof}

A Noetherian ring is called regular if so are all its localizations at its prime ideals (equivalently, at its maximal ideals). This leads us to the next theorem of Matijevic--Roberts type, also obtained in \cite[1.2.5]{GW}.

\begin{Theorem}\label{regular_r} Let $R$ be a Noetherian $\ZZ^r$-graded ring.  Then the following conditions are equivalent:
\begin{enumerate}
\item[(a)] $R$ is regular.
\item[(b)] For all prime ideals $\pp$, the local ring $R_{\pp^*}$ is regular.
\end{enumerate}
\end{Theorem}

We note some properties of  regular \*local rings that will be used later.

\begin{Proposition}\label{generators}
Let $R$ be a regular \*local ring with \*maximal ideal $\mm$. Then the following hold:
\begin{enumerate}
	\item[(a)] $R$ is an integral domain. 
	
	\item[(b)] $\mm$ is generated by an $R$-sequence of homogeneous elements $x_1,\dots,x_m$, $m = \*\dim R$.
	
	\item[(c)] The semigroup $\bigl\{\deg y: y\in R \text{ homogeneous} \bigr\}$ is generated by $\deg x_i$, $i = 1,\dots, m$, and the degrees of the homogeneous units.
\end{enumerate}
\end{Proposition}

\begin{proof}
All zero-divisors of $R$ are contained in $\mm$ by Lemma \ref{star}. Thus the natural map $R\to R_\mm$ is an injection, and the regular local ring $R_\mm$ is an integral domain. This proves (a).

For (b),we choose a minimal homogeneous system of generators $x_1,\dots,x_m$ of $\mm$. Then $x_1,\dots,x_m$ are a minimal system of generators of the maximal ideal $\mm R_\mm$ of the regular local ring $R_\mm$ by Theorem \ref{*local}. Thus $m = \dim R_\mm = \*\dim R$.

Since $R$ is an integral domain, $x_1$ is not a zero-divisor. We know already that $R/(x_1)$ is again regular \*local so that induction proves that  $x_1,\dots,x_m$ is a regular sequence.

For (c), we use induction on $m$. The case $m = 0$ is trivial. Let $m\ge 1 $. Since $R/(x_1)$ is again a regular \*local ring, it is a domain, and so $x_1$ generates a prime ideal in the ring $R$. Let $y\in R$ be a homogeneous element. It has a unique presentation $y = x_1^e z$ with an integer $e\ge 0$ and a homogeneous element $z$ that is not divisible by $x_1$. Passing to $R/(x_1)$, we see by induction that $\deg z$ lies in the semigroup generated by $\deg x_2,\dots, \deg x_m$ and the degrees of the homogeneous units. In order to get the degree of $y$, it is enough to add $\deg x_1$ to the generators.
\end{proof}

\section{The \*canonical module}\label{canonical}

Now it has become time to introduce a main  actor. The canonical module of an ordinary local Noetherian ring is developed in \cite[Section 3.3]{BH}. By definition it is a maximal Cohen--Macaulay module of finite injective dimension and type $1$. We may speak of \emph{the} canonical module of a local ring since it is uniquely determined up to isomorphism: if $\omega$ and $\omega'$ are canonical modules, then $\Hom_R(\omega,\omega')$ is a free module of rank $1$, and a basis element of this free module is an isomorphism. The canonical module localizes: 
$\omega_\pp$ is the canonical module of $R_\pp$ for all prime ideals $\pp$ of $R$.

An $R$-module $\omega$ is a canonical module of a general Noetherian ring if its localization $\omega_\pp$ is a canonical module of $R_\pp$ for all  prime ideals of $R$, equivalently, for all maximal ideals $\pp$. The canonical module is in general no longer uniquely determined up to isomorphism, but $\Hom_R(\omega,\omega')$ is a rank $1$ projective module. This follows immediately from the local case.

Let $R$ be a \*local ring with \*maximal ideal $\mm$. Note that $R/\mm$ has a natural grading as a residue class \emph{ring} of $R$, and this is the grading used in the following definition. 

\begin{Definition}
Let $(R, \mm)$ be a Cohen--Macaulay \*local ring of
\*dimen\-sion $d$. A finite graded $R$-module $\omega$ is a
{\em \*canonical module\/} of $R$ if there exist homogeneous
isomorphisms
$$
\*\Ext^{i}_{R} (R/\mm ,\omega) \iso
\begin{cases}
0&  i \neq d, \\
R/\mm  &  i = d.
\end{cases}
$$
\end{Definition}

For $\ZZ$-gradings, the following proposition is \cite[3.6.9]{BH}, except that we get rid of the superfluous assumption in (b) that the \*maximal ideal is an ordinary maximal ideal. The proof is almost verbatim the same.

\begin{Proposition}
\label {Uniqueness}
Let $(R,\mm)$ be a Cohen--Macaulay \*local ring, and 
$\omega$ be a \*canonical module of $R$. Then
\begin{enumerate}
\item[(a)] $\omega$ is a canonical module of $R$,
\item[(b)] $\omega$ is uniquely determined up to homogeneous isomorphism.
\end{enumerate}
\end{Proposition}

\begin{proof} For (a), we note first that $\omega_\mm$ is the canonical module of $R_\mm$ for the \*maximal ideal $\mm$. Of course, $\omega_\mm$ is Cohen--Macaulay and that it has type $1$ and finite injective dimension is bundled in the definition: the vanishing of the critical $\*\Ext_{R_\mm}^i(R_\mm/\mm R_\mm, \omega_\mm)$ for  $i\neq \dim R_\mm$ implies finite injective dimension, and the isomorphism to $R_\mm/\mm R_\mm$ means type $1$.

It remains to show that $\omega_\pp$ is the canonical module for all prime ideals $\pp$ of $R$. Since $\pp\subset \mm$ if $\pp$ is graded, the localization property of the canonical module yields the claim in this case. For general $\pp$ we use that $\omega_{\pp^*}$ is the canonical module of $R_{\pp^*}$ by what has just been said and that the characterizing  properties of the canonical module can be lifted from $\omega_{\pp^*}$ to $\omega_\pp$, as proved in Theorem \ref{depth-type}.

For the proof of (b), let $\omega'$ be another \*canonical module. As observed above, $\*\Hom_R(\omega, \omega')$ is a projective module of rank 1 (also in the nongraded situation),
and hence, as a graded module, is free; see Theorem \ref{*local}(d). Therefore,
$\*\Hom_R(\omega,\omega')\iso R(u)$ for some $u\in \ZZ^r$. This implies 
$\*\Hom_R(\omega,\omega'(-u))\iso R$. Let $\phi\in\*\Hom_R(\omega,\omega'(-u))$ be an
element corresponding to 1 under this identification. Then,
since $\*\Hom_R(\omega,\omega'(-u))=\Hom_R(\omega,\omega')$,  it follows from (a) that
$\phi$ is an isomorphism of $R$-modules. But then $\phi$ is a homogeneous isomorphism of graded $R$-modules, and  we have
\begin{align*}
R/\mm&\iso\*\Ext_R^d(R/\mm,\omega)\iso
\*\Ext_R^d(R/\mm,\omega'(-u))\\
&\iso\*\Ext_R^d(R/\mm,\omega')(-u)\iso (R/\mm)(-u).
\end{align*}
If $\mm$ is an ordinary maximal ideal, then $R/\mm$ is concentrated in degree $0$, and therefore $u=0$. In general we get that $(R/\mm)_u\neq 0$, and therefore $R$ contains a unit of degree $u$: see Lemma \ref{Laurent} for the structure of $R/\mm$. This implies that $\omega$ and $\omega'$ are isomorphic graded modules by Lemma \ref{iso}.
\end{proof}

\begin{Example}\label{regular}
(a) Let $R = k[t_1^{\pm 1},\dots, t_n^{\pm 1}]$ be a Laurent polynomial ring over a field $k$. Then $\*\dim R = 0$, and evidently $R$ is its own canonical module. 

(b) Let $R$ be a regular \*local ring and $x_1,\dots,x_m$ a regular homogeneous sequence generating the \*maximal ideal $\mm$. Then $R(-\sum_i \deg x_i)$ is the \*canonical module of $R$.

In fact, all homogeneous elements of $R/\mm$ are units, and $R/\mm$ is a Laurent polynomial ring over a field. By (a),  it is its own \*canonical module, and we can apply Rees' lemma \ref{Rees} successively to $x_1,\dots,x_m$ to get the \*canonical module of $R$. Take $N = R/\mm$ and $Q = R$ in Lemma  \ref{Rees}.

Alternatively one can compute $\*\Ext_R^m(R/\mm, R)$ from the Koszul complex for the sequence $x_1,\dots,\allowbreak x_m$ since the Koszul complex is a free resolution of $R/\mm$. (See \cite[Section 1.6]{BH} for an extensive treatment of the Koszul complex.)
\end{Example}

In the preceding proof, we have used Rees' lemma \ref{Rees} in order to compute the \*canonical module of the regular \*local ring $R$ from that of its residue class ring $R/\mm$. Going in the other direction one proves:

\begin{Proposition}\label{residue}
Let $R$ be a \*local Cohen--Macaulay ring with \*canonical module $\omega$, and $x_1,\dots, x_m$ a homogeneous regular $R$-sequence. Then $(\omega/(x_1,\dots,x_m)\omega)(\sum \deg x_i)$ is the \*canonical module of $R/(x_1,\dots,x_m)R$. 
\end{Proposition}

\begin{proof}
There is not much to say. Again we take $N$ as $R/\mm$ with the \*maximal ideal $\mm$ and observe that $\*\dim R/(x_1,\dots,x_m) = \*\dim R -m$.
\end{proof}

Example \ref{regular} and Proposition  \ref{residue} demonstrate the usefulness of change of rings in a special case, namely the passage to residue class rings modulo regular sequences. Below we will use  a change of rings argument that works in more general conditions.

We have already used that the ordinary canonical module $\omega$ of a Cohen--Macaulay ring $R$ localizes: for every prime ideal $\pp$ of $R$, the localization $\omega_\pp$ is the canonical module of $R_\pp$. The graded version remains true at least for residue class rings of regular \*local rings: an extra shift is not necessary. This was left open in \cite[3.6.12(a)]{BH}.

\begin{Theorem}\label{Local}
Suppose the \*local Cohen--Macaulay ring $R$ is the residue class ring of a regular \*local ring $S$. Let $c= \*\dim S - \*\dim R$.
\begin{enumerate}
\item[(a)] Then $\omega= \*\Ext_S^c(R,\omega_S)$ is the \*canonical module of $R$.

\item[(b)] $\omega_{(\pp)}$ is the \*canonical module of the homogeneous localization $R_{(\pp)}$ for every prime ideal $\pp$ of $R$ .
\end{enumerate}
\end{Theorem}

\begin{proof}
(a) Let $\nn$ and $\mm$ be the \*maximal ideals of $S$ and $R$, respectively. Then $S/\nn \iso R/\mm$ where the isomorphism is induced by the natural surjection $S\to R$. There is a change of rings spectral sequence
$$
\*\Ext_R^p(R/\mm,\*\Ext_S^q(R,\omega_S))\mathrel{\mathop{\Longrightarrow}\limits_p}
\*\Ext_S^{p+q}(S/\nn,\omega_S);
$$
see \cite[11.66]{Ro} for the nongraded case. By hypothesis, $R$ is a Cohen--Macaulay $S$-module. Therefore $\*\Ext_S^k(R,\omega_S) = 0$ for all $k\neq c$. This implies that the nonzero $E_2^{p,q}$
terms are concentrated in a single column and the spectral sequence
collapses. Setting $\omega=\*\Ext_S^c(R,\omega_S)$, $d = \*\dim R$,  we get
$$
\*\Ext_R^d(R/\mm, \omega) = \*\Ext_S^e(S/\nn, \omega_S) = S/\nn, \qquad e = \*\dim S = d + c.
$$
whereas $\*\Ext_R^k(R/\mm, \omega) = 0$ for all $k\neq d$.

(b) Let $\qq$ be the preimage of $\pp$ in $s$. Then $R_{(\pp)}$ is a residue class ring of $s_{(\qq)}$. By (a), we have
$$
\omega_{R_{(\pp)}} \iso \*\Ext^{c}_{S_{(\qq)}} (R_{(\pp)}, \omega_{S_{(\qq)}}).
$$
It is enough to show $(\omega_S)_{(\qq)} \iso \omega_{S_{(\qq)}}$.

By Example \ref{regular}, 
$$
\omega_S = S\Bigl(-\sum_{i=1}^g \deg x_i\Bigr),
$$
where $x_1,\dots,x_g$ is a homogeneous regular sequence generating the \*maximal ideal $\nn$ of $S$. We can assume that $x_1,\dots,x_m\in\qq$ and $x_{m+1}, \dots, x_g\notin\qq$. Since $x_{m+1}, \dots, x_g\notin\qq$ are units in $S_{(\qq)}$, shifts by their degrees are homogeneous isomorphism over $S_{(\qq)}$. Therefore
$$
(\omega_S)_{(\qq)} = S_{(\qq)}\Bigl(-\sum_{i=1}^m \deg x_i\Bigr).
$$

Set $h = \dim S_\qq$. Then one can complement $x_1,\dots, x_m$ by homogeneous elements $y_{m+1},\dots, y_h$ to a minimal system of generators of $\qq S_\qq$. One has
$$
\omega_{S_{(\qq)}} = S_{(\qq)}\Bigl(-\sum_{i=1}^m \deg x_i - \sum_{j=m+1}^h \deg y_j\Bigr ).
$$
Since the $y_j$ are not divisible by any $x_1,\dots,x_m$, it follows from Proposition \ref{generators} that their degrees are degrees of units. Thus shifts by them are irrelevant. Hence
$$
\omega_{S_{(\qq)}} = S_{(\qq)}\Bigl(-\sum_{i=1}^m \deg x_i\Bigr),
$$
and we are done.
\end{proof}

\section{The canonical module of a normal affine monoid ring}\label{DanStan}

For the theory of affine monoid rings we refer the reader to Bruns and Gubeladze \cite{BG}. We recapitulate the most important facts. An affine monoid $M$ is a finitely generated submonoid (or subsemigroup) of the abelian group $\ZZ^r$ for some $r\ge 0$. Choosing a field $K$ of coefficients, one obtains a commutative finitely generated $K$-algebra $K[M]$. The $K$-basis of $K{M}$ are monomials $X^x$ with exponent vectors $x \in M$.

The $K$-algebra $K[M]$ is normal, i.e., integrally closed in its quotient field, if and only if there is a subgroup $L$ of $\ZZ^r$ and a rational polyhedral cone $C$ such that $M = C \cap L$. It is not a restriction of generality if we assume that $\ZZ^r$ is generated by $M$ as a group. The cone $C$ is the intersection of finitely many half-spaces $H_i$ $i=1,\dots, s$, each of which is defined by an integral linear form $\lambda_i$: $H_i = \bigl\{x\in \RR^r: \lambda_i(x) =0 \}$. The $\lambda_i$ are unique up to scalar multiples (because of our assumption that $M$ generates $\ZZ^r$) , and unique if we choose coprime coefficients for them. 

By Hochster's theorem \cite[6.3.5]{BH}, $K[M]$ is a Cohen--Macaulay ring. For an application of Theorem \ref{Local}, we must write $K[M]$ as the residue class ring of a \*local $\ZZ^r$-graded regular ring. To this end, we set $N =\{x \in M: \lambda_i(x) = 0,\ i=1,\dots,s\}$ and let $T$ be the $K$-algebra $K[N]$. Since $-x\in N$ if and only $x\in N$, $T$ is a Laurent polynomial ring. Let $x_1,\dots,x_m$ be a minimal system of monomials that together with $N$ generates $M$. For indeterminates $Y_1,\dots,Y_m$, the $K$-algebra $K[M]$ is a residue class ring of $S = T[Y_1,\dots,Y_m]$ if we send $Y_i$ to $x_i$. The $ \ZZ^r$-grading on $S$ is obtained by pulling the $\ZZ^r$-grading of $K[M]$ back to $S$. It follows that $S$ is a $\ZZ^r$-graded \*local ring with \*maximal ideal $(Y_1,\dots,Y_m)$. 

Let $\pp_i$ be the $K$-vector space generated by the monomials  $X^y\in M$ with $\lambda_i(y) > 0$, equivalently $\lambda_i(y) \ge 1$. Then $\pp_1,\dots,\pp_s$ are divisorial prime ideals in $K[M]$, and by Chouinard's theorem \cite[4.60]{BG} the $\ZZ^r$-graded divisorial ideals of $K[M]$ are exactly the fractional ideals $\bigcap_{i=1}^s \pp_i^{(h_i)}$ for $h_1,\dots,h_s\in \ZZ$. The symbolic power $\pp_i^{(h_i)}$ is the vector space spanned by the monomials $X^y$, $\lambda_i(y) \ge h_i$. In other words, every divisorial ideal is isomorphic to one of $\bigcap_{i=1}^s \pp_i^{(h_i)}$, and these are pairwise non-isomorphic as $\ZZ^r$-graded fractional ideals.

Now we have all prerequisites  for a divisorial proof of the Theorem of, independently, Danilov and Stanley:

\begin{Theorem}\label{DS_theo}
With the notation and hypotheses introduced, the *canonical module of $R = K[M]$ is the ideal generated by the monomials $X^x$ for which $x$ belongs to the interior of the cone $C$.
\end{Theorem}

\begin{proof}
The $K$-basis of $\pp_1\cap\dots\cap \pp_s$ are the monomials $X^x$ for which $x$ belongs to the interior of the cone $C$. The \*canonical module $\omega$ is $\ZZ^r$-graded and a divisorial ideal. Thus $\omega = \bigcap_{i=1}^s \pp_i^{(h_i)}$ for uniquely determined integers $h_i$. Since $\pp_j^{(h_j)}R_{(\pp_i)} = R_{(\pp_i)}$ for $i\neq j$, it follows from Theorem \ref{Local} that $\omega_{\pp_i} = \pp_i^{(h_i)}R_{(\pp_i)}$. But $R_{(\pp_i)}$ is isomorphic to a Laurent polynomial ring $K[V,W_1^{\pm 1},\dots, W_{r-1}^{\pm 1} ]$, whence $H_i = 1$. It only remains to observe that $\pp_1\cap\dots\cap \pp_s$ is the ``interior'' ideal.
\end{proof}

In \cite[6.3.5]{BH} Theorem \ref{DS_theo} and Hochster's theorem are proved together by the computation of local cohomology. A divisorial approach to Theoerem \ref{DS_theo} has been used in \cite{BG}. The treatment of the \*canonical module in Sections \ref{canonical} and \ref{DanStan} is more conceptual, however.


\begin{thebibliography}{1}
	
\bibitem {BG} W.~Bruns and J.~Gubeladze, \emph{Polytopes,
		rings and K-theory}. Springer, 2009.
	
\bibitem{BH}  W. Bruns and  J. Herzog, {\em Cohen--Macaulay
    rings}. Cambridge Studies in Advanced Mathematics {\bf 39},
    Cambridge University Press, 1993, Rev. ed. 1998.
\bibitem{FF}   R. Fossum and H.-B. Foxby, {\em The category of graded modules}, Math. Scand.
   35 (1974), 288--300.
   
\bibitem{GW} S. Goto and K. Watanabe, {\em On graded rings, II ($\ZZ^n$-graded rings)}, Tokyo J. Math. 1 (1978), 237--261. 

\bibitem{HK}
J. Herzog and E. Kunz, Eds., {\em Der kanonische Modul eines Cohen--Macaulay-Rings}. LNM 238, Springer, 1971.

\bibitem{MR} J. Matijevic and P. Roberts, {\em A conjecture of Nagata on graded Cohen--Macaulay rings},  J. Math. Kyoto Univ. 14 (1974), 125--128.
\bibitem{Ro}  J.  Rotman, {\em An introduction to homological algebra}. Academic Press, 1979. 
\end{thebibliography}
\end{document}